\documentclass[reqno]{amsart}
\usepackage{amsfonts}
\usepackage{amsmath}
\usepackage{amsthm, amscd}
\usepackage{mathtools}
\usepackage{amssymb}
\usepackage{mathrsfs}
\usepackage{graphicx}
\usepackage{bbm}

\usepackage[alphabetic]{amsrefs}
\usepackage{etex}
\usepackage{hyperref}
\hypersetup{
  colorlinks = true,
  linkcolor  = black
}

\usepackage{xypic}

\usepackage{fancyhdr}
\usepackage{color} 
\usepackage{overpic}
\usepackage{tikz-cd}

\allowdisplaybreaks

\theoremstyle{plain}\newtheorem{Theorem}{Theorem}[section]
\theoremstyle{plain}
\theoremstyle{plain}\newtheorem{Lemma}[Theorem]{Lemma}
\theoremstyle{plain}
\theoremstyle{plain}\newtheorem{Proposition}[Theorem]{Proposition}
\theoremstyle{plain}
\theoremstyle{plain}\newtheorem{Quesion}[Theorem]{Question}
\theoremstyle{plain}\newtheorem*{Theorem*}{Theorem}
\theoremstyle{remark}\newtheorem{remark}[Theorem]{Remark}

\theoremstyle{plain}\newtheorem{thm}[Theorem]{Theorem}
\theoremstyle{plain}\newtheorem{cor}[Theorem]{Corollary}
\theoremstyle{plain}\newtheorem{prop}[Theorem]{Proposition}
\theoremstyle{plain}\newtheorem{lem}[Theorem]{Lemma}
\theoremstyle{plain}
\theoremstyle{plain}
\theoremstyle{plain}
\theoremstyle{plain}
\theoremstyle{plain}

\makeatletter
\newtheorem*{rep@theorem}{\rep@title}
\newcommand{\newreptheorem}[2]{%
\newenvironment{rep#1}[1]{%
 \def\rep@title{#2 \ref{##1}}%
 \begin{rep@theorem}}%
 {\end{rep@theorem}}}
\makeatother
\newreptheorem{theorem}{Theorem}

\theoremstyle{definition}
\newtheorem{defn}[Theorem]{\textbf{Definition}}

\newtheorem{exmp}[Theorem]{Example}

\theoremstyle{remark}
\newtheorem{rem}[Theorem]{Remark}

\numberwithin{equation}{section}

\DeclareMathOperator{\rank}{rank}

\DeclareMathOperator{\Kh}{Kh}
\DeclareMathOperator{\Khr}{Khr}

\DeclareMathOperator{\pt}{pt}

\DeclareMathOperator{\lk}{lk}
\DeclareMathOperator{\HFK}{\widehat{HFK}}
\DeclareMathOperator{\HFL}{\widehat{HFL}}
\DeclareMathOperator{\SFH}{SFH}

\newcommand{\bQ}{\mathbb{Q}}

\newcommand{\bZ}{\mathbb{Z}}

\newcommand{\bF}{\mathbb{F}}

\newcommand{\al}{\alpha}

\newcommand{\be}{\beta}
\newcommand{\ga}{\gamma}

\newcommand{\bpf}{\begin{proof}}
\newcommand{\epf}{\end{proof}}
\newcommand{\bthm}{\begin{thm}}
\newcommand{\ethm}{\end{thm}}
\newcommand{\bprop}{\begin{prop}}
\newcommand{\eprop}{\end{prop}}
\newcommand{\bcor}{\begin{cor}}
\newcommand{\ecor}{\end{cor}}
\newcommand{\blem}{\begin{lem}}
\newcommand{\elem}{\end{lem}}
\newcommand{\bdefn}{\begin{defn}}
\newcommand{\edefn}{\end{defn}}
\newcommand{\bexmp}{\begin{exmp}}
\newcommand{\eexmp}{\end{exmp}}
\newcommand{\brem}{\begin{rem}}
\newcommand{\erem}{\end{rem}}

\newcommand{\bdia}{\begin{displaymath}\xymatrix}
\newcommand{\edia}{\end{displaymath}}
\newcommand{\beq}{\begin{equation*}\begin{aligned}}
\newcommand{\eeq}{\end{aligned}\end{equation*}}

\newcommand{\intg}{\mathbb{Z}}

\author{Zhenkun Li}
\address{Department of Mathematics, Massachusetts Institute of Technology, Massachusetts 02139, USA}
\email{zhenkun@mit.edu}
\author{Yi Xie}
\address{Beijing International Center for Mathematical Research, Peking University, Beijing 100871, China}
\email{yixie@pku.edu.cn}
\author{Boyu Zhang}
\address{Department of Mathematics, Princeton University, New Jersey 08544, USA}
\email{bz@math.princeton.edu}
\title{Two detection results of Khovanov homology on links}

\begin{document}

\begin{abstract}
We prove that Khovanov homology with $\bZ/2$--coefficients detects the link L7n1, and
the union of a trefoil and its meridian.
\end{abstract}

\maketitle

\section{Introduction}

Given an oriented link $L$ in $S^3$ and a commutative ring $R$, Khovanov homology \cite{Kh-Jones}
assigns a bi-graded $R$--module $\Kh(L; R)$ to the link $L$. 
In 2011, Kronheimer and Mrowka \cite{KM:Kh-unknot} proved that Khovanov homology detects the unknot. Since then, many other detection results of Khovanov homology have been obtained. It is now known that Khovanov homology detects the unlink \cite{Kh-unlink, HN-unlink}, the trefoil \cite{BS}, the Hopf link \cite{BSX}, the forest of unknots \cite{XZ:forest}, the splitting of links \cite{LS-split}, and the torus link  $T(2,6)$ \cite{Martin:T26}.

In \cite{XZ:rank8}, a classification is given for all links $L$ such that  $\rank_{\bZ/2}\Kh(L;\bZ/2)\le 8$ and all 3-component links $L$ such that  $\rank_{\bZ/2}\Kh(L;\bZ/2)\le 12$. 
 By \cite[Corollary 3.2.C]{Shu:torsion_Kh},
$\rank_{\bZ/2}\Kh(L;\bZ/2)=2\rank_{\bZ/2}\Khr(L;\bZ/2)$, where $\Khr$ denotes the reduced Khovanov homology. Moreover, the parity of $\rank_{\bZ/2}\Khr(L;\bZ/2)$ is invariant under crossing changes and hence is always even for 2-component links (as is the case for the 2-component unlink). Therefore $\rank_{\bZ/2}\Kh(L;\bZ/2)$ is always a multiple of 4. As a consequence, if a 2-component link $L$ satisfies  $\rank_{\bZ/2}\Kh(L;\bZ/2)> 8$, then $\rank_{\bZ/2}\Kh(L;\bZ/2)\ge 12$. 

This paper studies 2-component links $L$ such that $\rank_{\bZ/2}(L;\bZ/2)=12$. Among 2-component links with crossing numbers less than or equal to 7, there are four links (up to mirror images) satisfying $\rank_{\bZ/2}(L;\bZ/2)=12$.  These links are: 
\begin{enumerate}
	\item L7n1 in the Thistlethwaite Link Table, 
	\item  L6a3 in the Thistlethwaite Link Table, 
	\item the disjoint union of a trefoil and an unknot,  
	\item the union of a trefoil and its meridian.
\end{enumerate}

\begin{Quesion}
    Suppose $L$ is a 2-component link with $\rank_{\bZ/2}(L;\bZ/2)=12$, is it true that $L$ must be isotopic (up to mirror image) to one of the links listed above?
\end{Quesion}

Instead of giving a full answer to the question above, we show that Khovanov homology (with the bi-grading)  detects the link L7n1, and the union of a trefoil with its meridian, from the list above. Since \cite{XZ:forest} proved that Khovanov homology detects the disjoint union of a trefoil and an unknot, and Martin \cite{Martin:T26} recently proved that Khovanov homology detects L6a3, we conclude that Khovanov homology detects all the links on the list.

In the following, we will call the link L7n1 as $L_1$, and the union of a trefoil with a meridian $L_2$. Moreover, we fix the chirality and orientation of these two links by 
Figure \ref{fig_L1_and_L2}. Notice that the link $L_1$ can also be described as 
 the closure of the 2-braid $\sigma_1^3$
together with an axis unknot.

\begin{figure}
\centering    
\begin{overpic}[width=\textwidth]{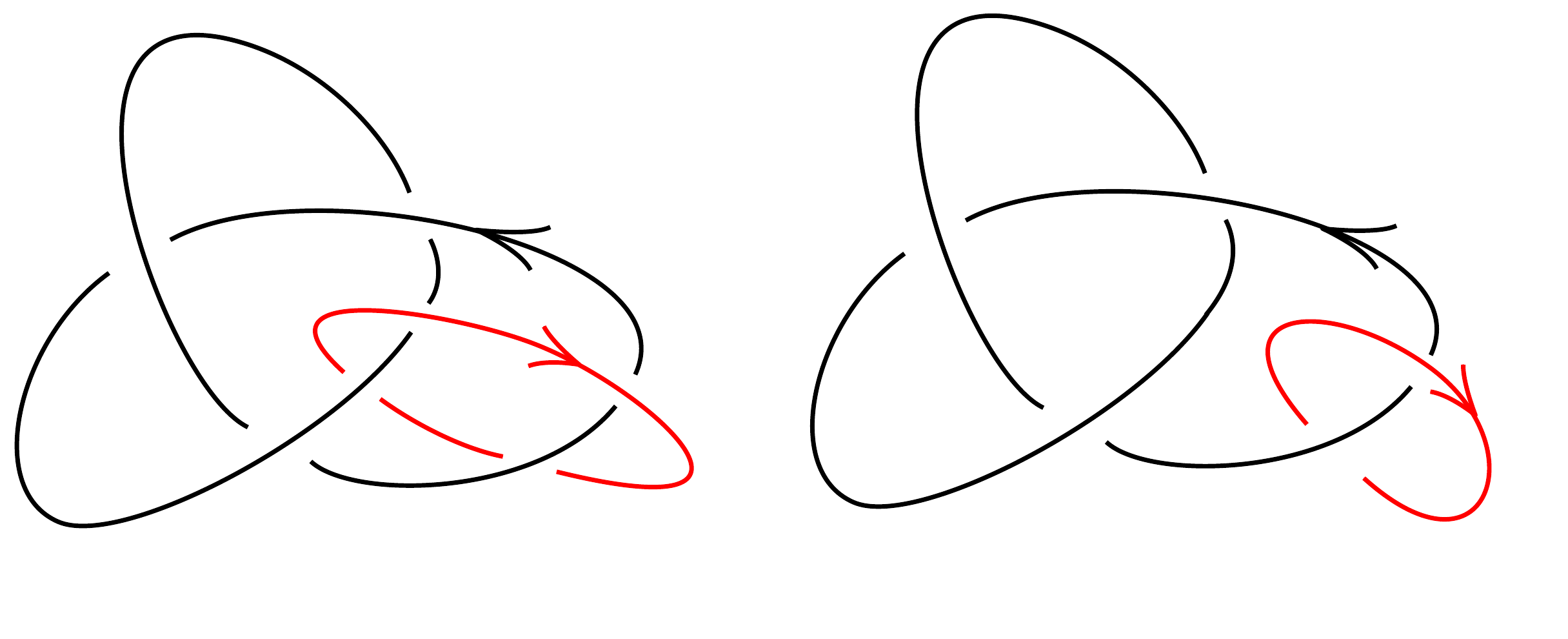}
    \put(13,0){$L_1={\rm L7n1}$}
    \put(58,0){$L_2={\rm trefoil~}\cup{~\rm meridian}$}
    \end{overpic}
\caption{The two links $L_1$ and $L_2$}\label{fig_L1_and_L2}
\end{figure}

 Recall that the internal grading of Khovanov homology is defined by $h-q$ 
 in \cite{Kh-unlink}, where 
$h$ is the homological grading and $q$ is the  quantum grading. The precise statement of our detection result is given as follows. 

\begin{Theorem}
\label{thm_main_detection}
Let $L_1,L_2 \subset S^3$ be the oriented links as shown in Figure \ref{fig_L1_and_L2}, and let $i\in\{1,2\}$. Suppose $L\subset S^3$ is a 2-component oriented link, such that 
$$\Kh(L;\bZ/2)\cong \Kh(L_i;\bZ/2)$$ as abelian groups equipped with the internal gradings, then $L$ is isotopic to $L_i$ as oriented links.
\end{Theorem}

The proof of Theorem \ref{thm_main_detection} depends on a rank inequality between reduced Khovanov homology and knot Floer homology by Dowlin \cite{Dowlin}, and a braid detection property of link Floer homology by Martin \cite{Martin:T26}. The main ingredient of the proof of Theorem \ref{thm_main_detection} is the following proposition, which is established in Section \ref{sec_rank_12_braid}.

\begin{Proposition}\label{prop_rank_12_braid}
Let $L=K\cup U$ be a link such that $U$ is an unknot and $K$ is either an unknot or a trefoil. Let $l=|\lk(K,U)|$ be the linking number of $K$ and $U$. Suppose $l>0$, and
\begin{equation}
\label{eqn_dim_assumption_HFK(L;Q)}
	\dim_\bQ \HFK(L;\bQ)\le 12,
\end{equation}
where $\HFK$ is the knot Floer homology defined in \cite{Ras:HFK,OS:HFK}.
Then at least one of the following holds:
\begin{enumerate}
	\item $K$ is the closure of an $l$-braid with axis $U$.
	\item $l=1$, $K$ is an unknot.
	\item  $l=1$, $K$ is a trefoil, $U$ is the meridian of $K$.
\end{enumerate}
\end{Proposition}

Recall that a 2-component link $K_1\cup K_2$ is said to be \emph{exchangeably braided} (or {\it mutually braided}) if both $K_1$ and $K_2$ are unknots, $K_1$ is a braid closure with axis $K_2$, and $K_2$ is a braid closure with axis $K_1$. The concept of exchangeably braided links was introduced and studied by Morton \cite{Mor-braid}. We will also need the following result from \cite{XZ:rank8}.

\begin{Proposition}[{\cite[Corollary 3.9]{XZ:rank8}}]
\label{prop_linking_number_3}
Suppose $L$ is an exchangeably braided link with linking number $l\ge3$,
then we have $
\rank_{\bQ}\HFK(L;\bQ)\ge 12.
$
Moreover, if $l>3$, then
$\rank_{\bQ}\HFK(L;\bQ)>12.$
\end{Proposition}

We will prove Theorem \ref{thm_main_detection} in Section \ref{sec_detection} as a consequence of Proposition \ref{prop_rank_12_braid}, Proposition \ref{prop_linking_number_3}, Dowlin's rank inequality \cite[Corollary 1.7]{Dowlin}, and Batson-Seed's spectral sequence \cite{Kh-unlink}.
\\

{\bf Acknowledgement.} The first author is supported by his advisor Tom Mrowka's NSF Grant 1808794.

\section{Link Floer homology}
\label{sec_link_floer}

This section reviews the basic properties of link Floer homology and proves a result on the rank of link Floer homology that will play an important role in the proof of Proposition \ref{prop_rank_12_braid}.

The link Floer homology was originally defined for ${\mathbb{Z}/2}$--coefficients by Ozsv\'ath and Szab\'o in \cite{OS:HFK}, and was generalized to $\bZ$--coefficients in \cite{Sar-HFL}. We will work with $\bQ$--coefficients in order to invoke Dowlin's spectral sequence \cite{Dowlin}. For the rest of this section, all Floer homology groups are with $\bQ$--coefficients and it will be omitted from the notation.

Given an oriented $n$--component link $L\subset S^3$,
its link Floer homology $\HFL(L)$ carries a homological grading over $\bZ$ and $n$ Alexander 
gradings associated to the $n$ components of $L$. The Alexander grading
 associated to the $i$-th component $K_i$ takes values in either $\bZ$ or $\bZ+\frac12$, which 
depends on the parity of the linking number $\lk(K_i,L-K_i)$.

By \cite{OS:HFL}, when $n\ge 2$, the link Floer homology recovers the multi-variable Alexander polynomial in the following sense:
\begin{multline}\label{eq_HFL_Alex}
\sum_{a_1,\cdots,a_n}\chi\big(\HFL(L,a_1,\cdots,a_n)\big) \cdot T_1^{a_1}\cdots T_n^{a_n}
\\
\doteq (T_1^{1/2}-T_1^{-1/2})\cdots 
(T_n^{1/2}-T_n^{-1/2})\Delta_L(T_1,\cdots,T_n),
\end{multline}
where $\HFL(L,a_1,\cdots,a_n)$ is the component of $\HFL(L)$ with multi-Alexander grading $(a_1,\cdots,a_n)$, and $\chi(\cdot)$ denotes the Euler characteristic with respect to the homological grading. The notation "$\doteq$" 
means that the two sides are equal up to a multiplication by $\pm T_1^{b_1}\cdots T_n^{b_n}$
for some $b_1,\cdots b_n\in \frac12 \bZ$. There is also a symmetry
\begin{equation}\label{eq_HFL_symmetry}
\HFL(L,a_1,\cdots,a_n)\cong \HFL(L,-a_1,\cdots,-a_n).
\end{equation}

The following proposition is a special case of the Thurston norm detection property of link Floer homology.
\begin{Proposition}[{\cite[Theorem 1.1]{OS:HFL_Thurston_norm}}]\label{prop_HFL_Thurston_norm}
Suppose $L=K\cup U\subset S^3$ is a 2-component link 
with an unknotted component $U$ and
$l=|\lk(K,U)|>0$. 
Then  
the top  Alexander grading of $\HFL(L)$ associated to $U$ is $\frac{l}{2}$ if and only
if $U$ has a Seifert disk that intersects $K$ transversely at  $l$ points.
\end{Proposition}

\brem
The proof of \cite[Theorem 1.1]{OS:HFL_Thurston_norm} was originally given for $\bZ/2$--coefficients, but the same argument applies to $\bQ$--coefficients. Alternatively, a similar norm-detection property for instanton Floer homology was established by \cite{li2019decomposition} using sutured manifold decompositions and the formal properties of Floer homology, and the same argument can be carried over to Heegaard Floer homology with $\bQ$--coefficients.
\erem

The following is a weaker version of a result from \cite{Martin:T26}.
\begin{Proposition}[{\cite[Corollary 2]{Martin:T26}}]\label{prop_braid_detection_HFL}
Let $L=K\cup U\subset S^3$ be a 2-component link such that $U$ an unknot and
$l=|\lk(K,U)|>0$. Then $K$ is the closure of a braid with axis $U$ if and only if 
the dimension of $\HFL(L)$ is 2 at the top Alexander grading associated to $U$.
\end{Proposition}

The link Floer homology $\HFL$ can be interpreted by sutured Floer homology using the following proposition. Here we use $\SFH$ to denote the sutured Floer homology defined by Juh\'asz in \cite{Juhasz-holo-disk}.

\bprop[{\cite[Proposition 9.2]{Juhasz-holo-disk}}]\label{lem: HKL and SFH}
Suppose $L=K_1\cup...\cup K_n$ is an oriented link, and let $S^3-N(L)$ be the link complement. Let $\ga$ be a suture on $\partial (S^3-N(L))$ which consists of two meridians of each $K_i$. Then there is an isomorphism
$$\HFL(L)\cong \SFH(S^3-N(L),\ga).$$
Moreover, the Alexander grading associated to $K_i$ corresponds to the grading induced by a Seifert surface of $K_i$ on $\SFH(S^3-N(L),\ga)$.
\eprop

\brem
The original statement is for $\intg/2$--coefficients, but the proof is done by examining the Heegaard diagrams, which also works for $\bQ$--coefficients.
\erem

We also need the following proposition from \cite{juhasz2010polytope}.
\bprop[{\cite[Proposition 9.2]{juhasz2010polytope}}]\label{prop: tensor Q^2}
Suppose $(M,\ga)$ is a balanced sutured manifold. Suppose $\ga_0$ is a component of $\ga$ that is homologically essential on $\partial M$. Let $\ga'$ be a suture on $\partial M$ obtained by adding two parallel copies of $\ga_0$ to $\ga$. Then we have
$$\SFH(M,\ga')\cong \SFH(M,\ga)\otimes_{\bQ}\bQ^2.$$
\eprop

The main result of this section is the following proposition.

\begin{Proposition}\label{prop_HFL_top_grading_parity}
Suppose $L=K\cup U\subset S^3$ is a 2-component link with an unknotted component $U$ and
$l=|\lk(K,U)|>0$, and suppose $U$ has a Seifert disk $D$ that intersects $K$ transversely at $l$ points.
Then 
$$
\dim_\bQ \HFL(L,\frac{l}{2})\equiv 2 ~{\rm mod}~4,
$$
 where $\HFL(L,l/2)$ is the component of $\HFL(L)$ with degree $l/2$ on the 
Alexander grading associated to $U$.
\end{Proposition}

In order to prove Proposition \ref{prop_HFL_top_grading_parity}, we need to establish the following property of sutured Floer homology.

\begin{Proposition}\label{thm: rank is odd for product tangle}
Let $l\in\bZ^+$, let $T\subset [-1,1]\times D^2$ be a tangle given by $T=\alpha_1\cup\cdots \cup \alpha_l$, where $\alpha_i$ is an arc connecting $\{-1\}\times D^2$ and $\{1\}\times D^2$ for all $i$. Let $M_T=[-1,1]\times D^2-N(T)$, let $\gamma_T\subset \partial M_T$ be a suture on $M_T$ with $(l+1)$ components: one meridian component
on each one of $\partial N(\al_1),\cdots  
\partial N(\al_{l})$, and a component on $[-1,1]\times \partial D^2$ given by $\{\pt\}\times \partial D^2$.
Then $\dim_{\bQ}\SFH(M_T,\ga_T)$ is odd.
\end{Proposition}

We start the proof of Proposition \ref{thm: rank is odd for product tangle} by verifying the trivial case.

\blem\label{lem: rank is odd for product tangle}
If $T$ is a product tangle, i.e., there are points $p_1,...,p_n\subset {\rm int}(D^2)$ so that $\al_i=[-1,1]\times \{p_i\}$ for all $i$, then $\dim_{\bQ}\SFH(M_T,\ga_T)$ is odd.
\elem
\bpf
When $T$ is a product tangle, $(M_T,\ga_T)$ is a product sutured manifold. Hence it follows from \cite{Juhasz-holo-disk} that the dimension of $\SFH(M_T,\ga_T)$ is one.
\epf

Let $T$ be the tangle in Proposition \ref{thm: rank is odd for product tangle}.
Orient $T$ so that each $\al_i$ goes from $\{-1\}\times D^2$ to $\{1\}\times D^2$. Fix a diagram on $[-1,1]\times [-1,1]$ that represents the tangle $T$. We will also denote the diagram by $T$ when there is no source of confusion. For a positive crossing of $T$, we can perform surgeries along the curve $\be$ as depicted in Figure \ref{fig: surgery along beta}. Let $M_{T,-1}$ be the manifold obtained by performing the $(-1)$--surgery along $\be$, and let $M_{T,0}$ be the manifold obtained by performing the $0$--surgery along $\be$. Let $T_-$ be the tangle that only differs from $T$ at the crossing linked by $\be$ as depicted in Figure \ref{fig: surgery along beta}. It straightforward to show that $(M_{T,-1},\ga_{T})\cong(M_{T_-},\ga_{T_-})$.

\begin{figure}
\centering    
\begin{overpic}[width=\textwidth]{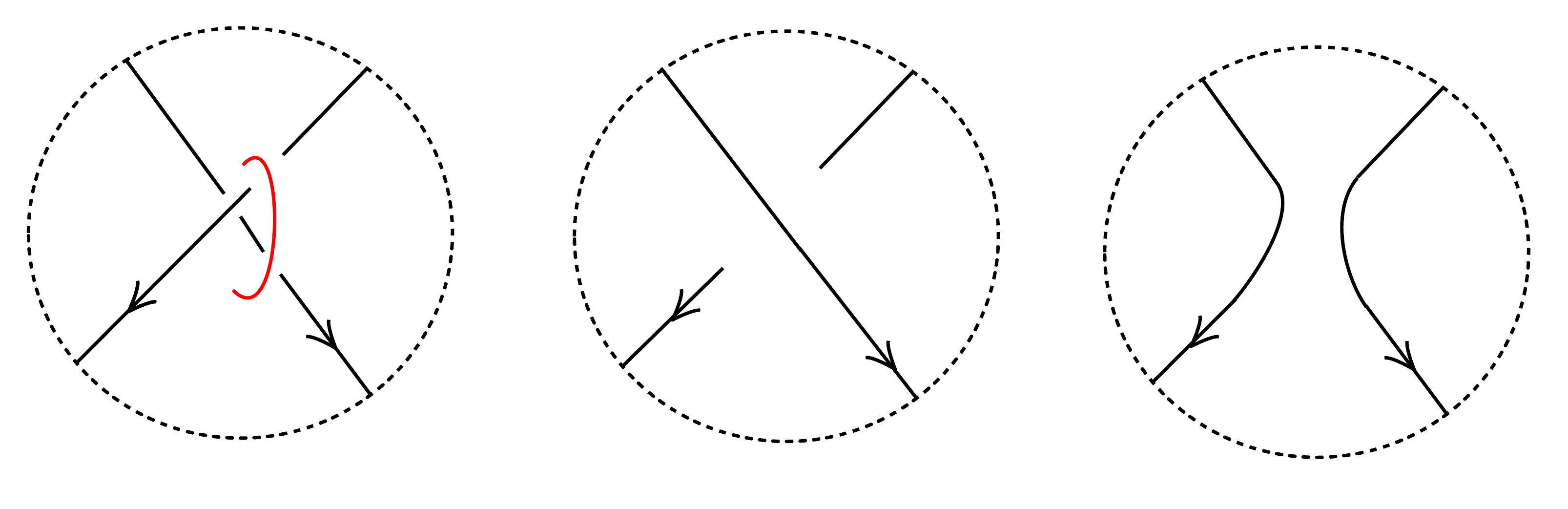}
    \put(19,18){$\be$}
    \put(13,0){$T$}
    \put(48,0){$T_-$}
    \put(83,0){$T_0$}
    \end{overpic}
\caption{Surgery along $\be$}\label{fig: surgery along beta}
\end{figure}
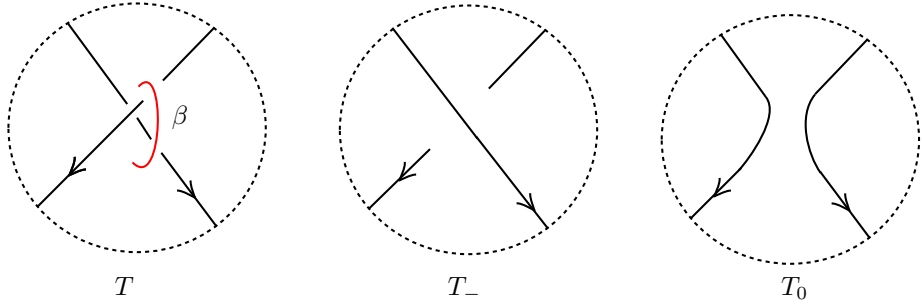

\bdefn\label{defn: crossing change}
We call the operation of switching from $T$ to $T_-$ or from $T_-$ to $T$ a {\it crossing change}.
\edefn

\blem\label{lem: crossing change}
For any vertical tangle $T\subset [-1,1]\times D^2$, there is a finite sequence of crossing changes that takes $T$ to the product tangle. \qed
\elem

Now we study the sutured manifold $(M_{T,0},\ga_{T})$. Inside $[-1,1]\times D^2$, the circle $\be$ bounds a disk $D$ that intersects the tangle $T$ twice. After performing the $0$--surgery, the boundary $\partial D$ can be capped by a meridian disk in the surgery solid torus, and hence we obtain a $2$--sphere $S$ that intersects the tangle $T$ twice. 
The intersection of $S$ and $M_{T,0}$ is a properly embedded annulus $A_{\be}\subset M_{T,0}$. We can pick the suture $\ga_{T}$ so that one boundary component of $A_{\be}$ lies in $R_+(\ga_{T})$ and the other lies in $R_-(\ga_{T})$. Then there is a sutured manifold decomposition
$$(M_{T,0},\ga_{T})\stackrel{A_{\be}}{\leadsto}(M',\ga').$$
From \cite[Section 3.1]{KM:Alexander}, we know that $M'\cong M_{T_0}:=[-1,1]\times D^2-N(T_0)$, where $T_0$ is another tangle on $[-1,1]\times D^2$, possibly having closed components, such that $T_0$ only differs from $T$ near the crossing linked by $\be$ as depicted in Figure \ref{fig: surgery along beta}.
\bdefn
We say that $T_0$ is obtained from $T_+$ by an oriented smoothing.
\edefn

\blem\label{lem: rank is even for T_0}
We have
$$\SFH(M_{T,0},\ga_{T})\cong \SFH(M_{T_0},\ga'),$$
and $\dim_{\bQ}\SFH(M_{T_0},\ga')$ is even.
\elem

\bpf
The isomorphism
$$\SFH(M_{T,0},\ga_{T})\cong \SFH(M_{T_0},\ga')$$
follows from \cite[Lemma 8.9]{Juh:sut}. For the parity statement, we argue in two cases.

{\bf Case 1}. The crossing linked by $\be$ involves two different components of $T$. Without loss of generality we can assume that they are $\al_1$ and $\al_2$. See Figure \ref{fig: oriented smoothing, 1}. Recall we have an annulus $A_{\be}\subset M_{T,0}$ after performing the $0$-surgery along $\be$. To make sure that the two boundary components of $A_{\be}$ lie in two different component of $R(\ga_{T})$, the suture $\ga_{T}$ must be arranged as in one of the two possibilities shown in Figure \ref{fig: oriented smoothing, 1}. After performing the sutured manifold decomposition along $A_{\be}$, the new tangle $T_0$ has two new arcs $\al_1'$ and $\al_2'$. For $i=1,2$, let $C_i'=\partial N(\alpha'_i)-\{-1,1\}\times D^2$. It is straightforward to check that, after the sutured manifold decomposition along $A_{\be}$, one and exactly one of the following two possibilities happens, as shown in Figure \ref{fig: oriented smoothing, 1}:
\begin{itemize}
\item $\ga'\cap C_1'$ consists of three parallel copies of meridians of $\al_1'$
\item $\ga'\cap C_2'$ consists of three parallel copies of meridians of $\al_2'$.
\end{itemize}

\begin{figure}
\centering    
\begin{overpic}[width=\textwidth]{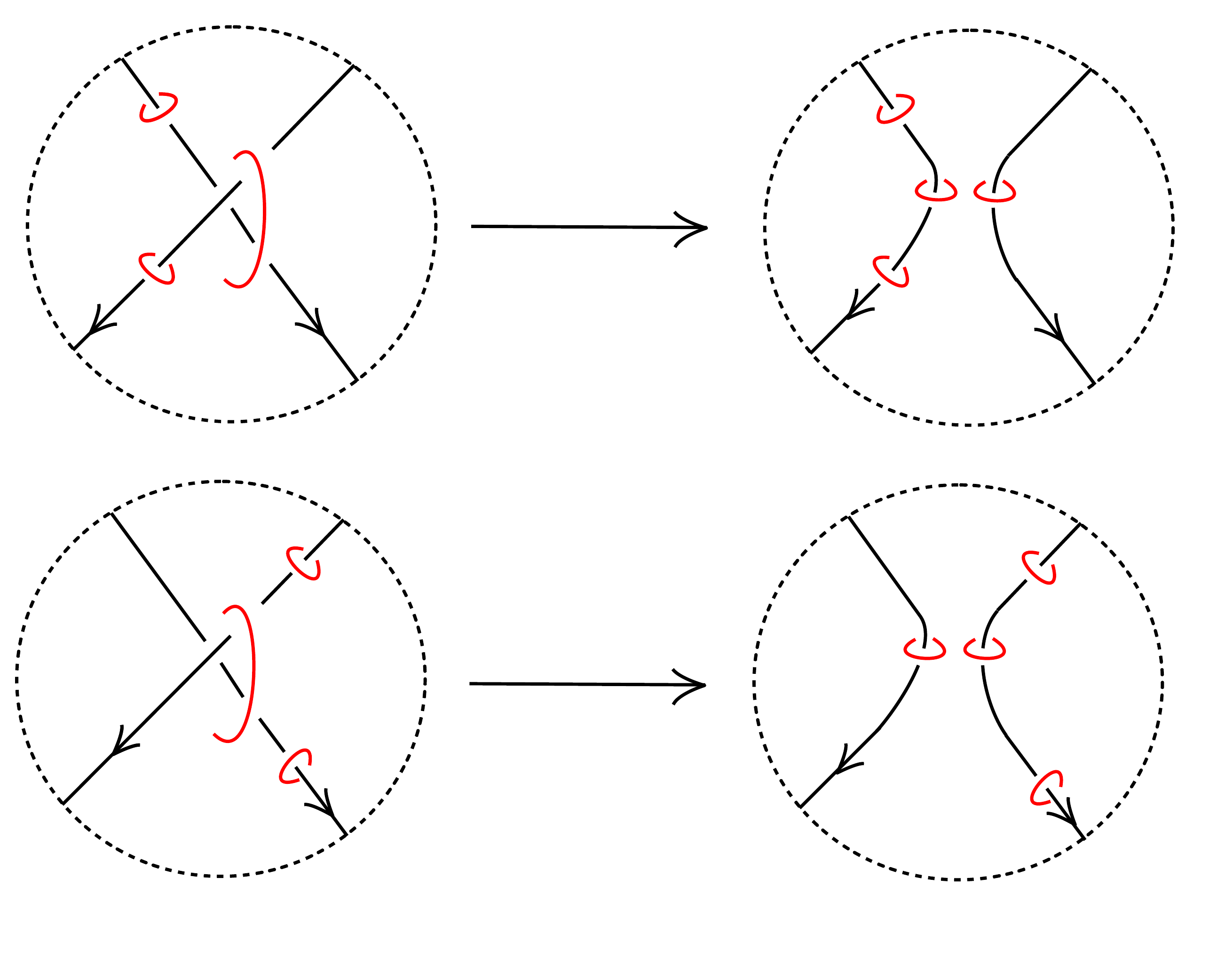}
    \put(22,25){$\be$}
    \put(26,31){$\ga_{T}$}
    \put(26,20){$\ga_{T}$}
    \put(23,62){$\be$}
    \put(7,70){$\ga_{T}$}
    \put(78,68){$\ga'$}
    \put(78,31){$\ga'$}
    \put(7,60){$\ga_{T}$}
    \put(16,4){$T$}
    \put(2,13){$\al_1$}
    \put(29,10){$\al_2$}
    \put(62,13){$\al'_1$}
    \put(89,10){$\al'_2$}
    \put(2,50){$\al_1$}
    \put(29,47){$\al_2$}
    \put(62,50){$\al'_1$}
    \put(89,47){$\al'_2$}
    \put(40,27){Decompose}
    \put(40,64){Decompose}
    \put(76,4){$T_0$}
    \end{overpic}
\caption{Oriented smoothing}\label{fig: oriented smoothing, 1}
\end{figure}
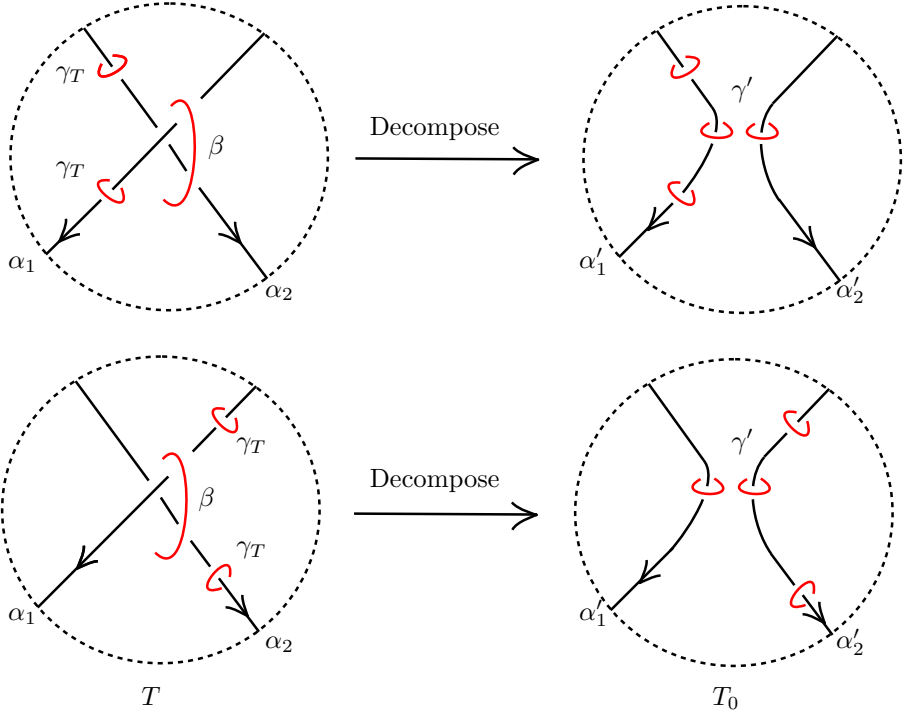

Without loss of generality, we assume that the first possibility happens, i.e., $\ga'$ contains three copies meridians of $\al'_1$. Removing two such copies, we obtain a new sutured manifold $(M_{T_0},\ga_{T_0})$, and by Proposition \ref{prop: tensor Q^2} we have
$${\rm dim}_{\bQ}\SFH(M_{T_0},\ga')=2\,{\rm dim}_{\bQ}\SFH(M_{T_0},\ga_{T_0}).$$
As a result, $\dim_{\bQ}\SFH(M_{T_0},\ga')$ is even.

{\bf Case 2}. The crossing linked by $\be$ involves only one component of $T$. Without loss of generality, we assume it is $\al_1$, see Figure \ref{fig: oriented smoothing, 2}. Let 
$C_i=\partial N(\al_i)-\{-1,1\}\times D^2$ for all $i$.  To make sure that the two boundary components of $A_{\be}$ lie in two different components of $R(\ga_{T})$, we must have the suture on $C_1$ to be in the position as depicted in Figure \ref{fig: oriented smoothing, 2}. After the decomposition along $A_\be$, the new tangle $T_0$ now has a closed component, which we call $\al_0'$, and an arc that we call $\al_1'$. Let $C_0'=\partial N(\al_0')$ and $C_1'=\partial N(\al_1')- \{-1,1\}\times D^2$. Note that $C_0'$ is a torus while $C_1'$ is an annulus. The suture $\ga'$ contains two meridians on $C_0'$, and one meridian on $C_1'$ (and one meridian on every other $C_i$).
\begin{figure}
\centering    
\begin{overpic}[width=\textwidth]{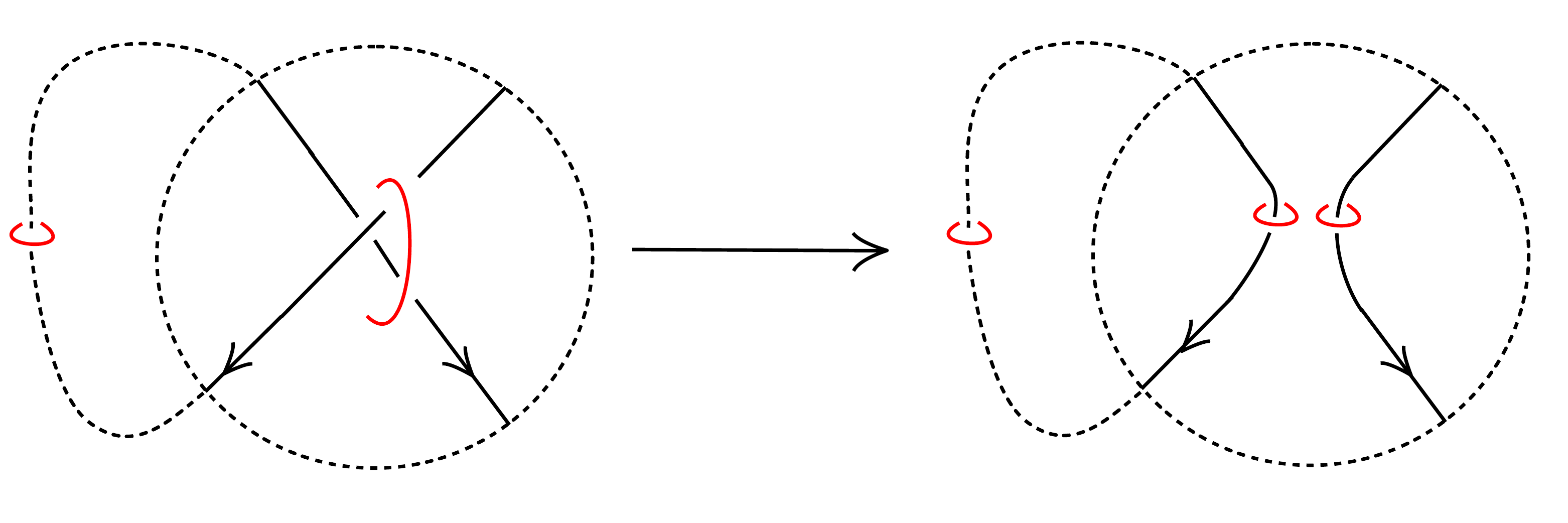}
    \put(32,29){$\al_1$}
    \put(92,29){$\al_1'$}
    \put(27,18){$\be$}
    \put(4,18){$\ga_{T}$}
    \put(64,18){$\ga'$}
    \put(77,19){$\ga'$}
    \put(88,19){$\ga'$}
    \put(74,31){$\al_0'$}
    \put(23,0){$T$}
    \put(40,20){Decompose}
    \put(83,0){$T_0$}
    \end{overpic}
\caption{Oriented smoothing}\label{fig: oriented smoothing, 2}
\end{figure}
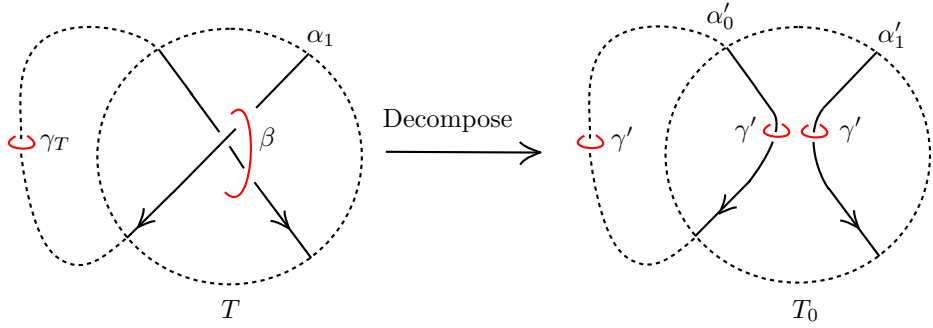

Recall that our goal is to show that $\dim_\bQ\SFH(M_{T_0},\ga')$ is even. Write
$$T_0'=T_0\backslash\al'_0.$$

{\bf Case 2.1}. When $\al_0'$ is split from $T_0'$, i.e., there is a $3$-ball $B^3\subset (-1,1)\times D^2$ so that
$$B^3\cap T_0=\al_0'.$$
In this case, we know that $(M_{T_0},\ga')$ is a connected sum:
$$(M_{T_0},\ga')\cong(M_{T_0'},\ga'- B^3)\#(S^3(\al_0'),\ga'\cap B^3).$$
Here $S^3(\al_0')$ is the knot complement of the knot $\al_0'\subset B^3\subset S^3$. It then follows from \cite[Proposition 9.15]{Juhasz-holo-disk} that the dimension of $\SFH(M_{T_0},\ga')$ is even.

{\bf Case 2.2}. When $\al_0'$ is not split from $T_0'$. Suppse there is a positive crossing of $T_0$ involving both $\al_0'$ and $T_0'$. Pick the circle  $\theta$ as depicted in Figure \ref{fig: surgery along theta}. Suppose the component of $T_0'$ involved in the crossing is $\al'$. There is a surgery exact triangle associated to $\theta$:
\begin{equation*}
    \xymatrix{
    \SFH(M_{T_0},\ga')\ar[rr]&&\SFH(M_{T_{0,-}},\ga'_-)\ar[dl]\\
    &\SFH(M_{T_{0,0}},\ga'_0)\ar[lu]&
    }
\end{equation*}

\begin{figure}
\centering    
\begin{overpic}[width=\textwidth]{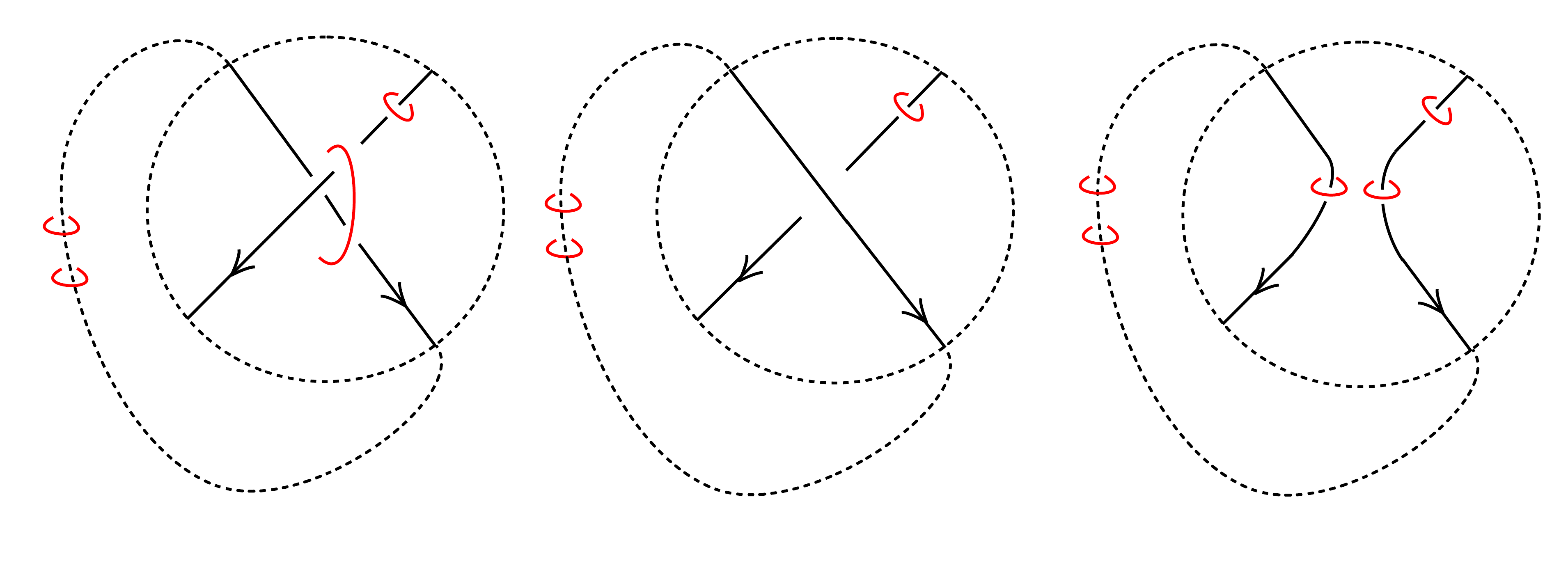}
    \put(23.5,23){$\theta$}
    \put(26,26.5){$\ga'$}
    \put(6,20){$\ga'$}
    \put(18,0){$T_0$}
    \put(28,10){$\al_0'$}
    \put(28,32){$\al'$}
    \put(61,10){$\al_0'$}
    \put(62,32){$\al'$}
    \put(94,10){$\al_0''$}
    \put(52,0){$T_{0,-}$}
    \put(86,0){$T_{0,0}$}
    \end{overpic}
\caption{Surgery along $\theta$}\label{fig: surgery along theta}
\end{figure}
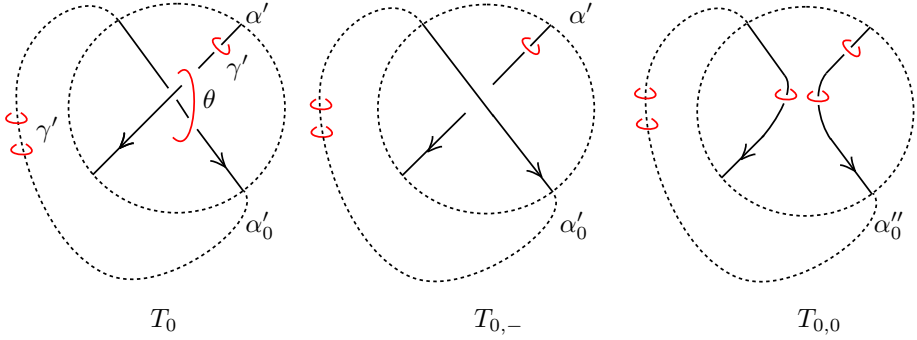

As above, $T_{0,-}$ is obtained from $T_0$ by a crossing change and $(M_{T_{0,-}},\ga_-')$ is the corresponding sutured manifold. The tangle $T_{0,0}$ is obtained from $T_0$ by an oriented smoothing. As in Figure \ref{fig: surgery along theta}, $\al_0'$ and $\al'$ merge into a single component $\al''_0\subset T_{0,0}$. It is then straightforward to check that the new suture $\ga'_0$ consists of five meridians of $\al''_0$: The two meridians of $\al'_0$ and one meridian of $\al'$ all survive, and there are two more meridians coming from the decomposition along an annulus $A_{\theta}$ (similar to the annulus $A_{\be}$ above). By Proposition \ref{prop: tensor Q^2}, we know that the dimension of $\SFH(M_{T_{0,0}},\ga_0')$ is even, and hence
\begin{equation}\label{eq: same parity}
    \dim_{\bQ}\SFH(M_{T_0},\ga')\equiv\dim_{\bQ}\SFH(M_{T_{0,-}},\ga'_-)~{\rm mod}~2.
\end{equation}

However, $T_0$ and $T_{0,-}$ only differ by a crossing change, and following the same line of Lemma \ref{lem: crossing change}, there is a finite sequence of such crossing changes that makes $\al_0'$ split from $T_0'$. Hence, it follows from Case 2.1 and (\ref{eq: same parity}) that the dimension of $\SFH(M_{T_0},\ga')$ must be even. This concludes the proof of Lemma \ref{lem: rank is even for T_0}.
\epf

\bpf[Proof of Proposition \ref{thm: rank is odd for product tangle}]
There is a surgery exact triangle associated to $\be$:
\begin{equation*}
    \xymatrix{
    \SFH(M_{T},\ga_{T})\ar[rr]&&\SFH(M_{T_-},\ga_{T_-})\ar[dl]\\
    &\SFH(M_{T_0},\ga')\ar[lu]&
    }
\end{equation*}
Therefore Lemma \ref{lem: rank is even for T_0} implies $\dim_{\bQ}\SFH(M_{T},\ga_{T})\equiv \dim_{\bQ}\SFH(M_{T_-},\ga_{T_-})~\rm{mod}~2$.  Proposition \ref{thm: rank is odd for product tangle} then follows from Lemma \ref{lem: rank is odd for product tangle} and Lemma \ref{lem: crossing change}.
\epf

\begin{remark}
    The statement and the proof of Proposition \ref{thm: rank is odd for product tangle} can be applied to sutured monopole theory and sutured instanton theory as well (with suitable choices of coefficients).
\end{remark}

\begin{proof}[Proof of Proposition \ref{prop_HFL_top_grading_parity}]
Take a pair of oppositely oriented meridional sutures to each boundary component
of $S^3-N(L)$, then $S^3-N(L)$ becomes a balanced sutured manifold.

Decompose $S^3-N(L)$ along the disk $D$, we obtain a sutured manifold $(M,\gamma)$. The manifold $M$ is given by $[-1,1]\times D^2-N(T)$, where $T$ is a tangle in $[-1,1]\times D^2$. Since the linking number of $K$ and $U$ is equal to $|K\cap U|$, we have $T=\alpha_1\cup\cdots \cup \alpha_l$ where $\alpha_i$ is an arc from $\{1\}\times D^2$ to $\{-1\}\times D^2$ for each $i$. The suture $\gamma$ consists of $(l+3)$ components: one meridian
on each of $\partial N(\al_1),\cdots  
\partial N(\al_{l-1})$, three parallel meridians on $\partial N(\al_l)$, and one component on $[-1,1]\times \partial D^2$ given by $\{\pt\}\times \partial D^2$. We have
$$
\HFL(L,\frac{l}{2})\cong \SFH(M,\gamma).
$$
 Removing two sutures from $\partial N(\al_l)$, we obtain the sutured manifold $(M_T,\gamma_T)$ as in Proposition \ref{thm: rank is odd for product tangle}.
By Proposition \ref{prop: tensor Q^2}, we have
$$
\dim_\bQ \SFH(M_T,\gamma)=2 \dim_\bQ \SFH(M_T,\gamma_T).
$$
Therefore the desired result follows from 
Proposition \ref{thm: rank is odd for product tangle}.
\end{proof}

\section{Proof of Proposition \ref{prop_rank_12_braid}}
\label{sec_rank_12_braid}
The strategy of our proof of Proposition \ref{prop_rank_12_braid} is to exploit the properties of the multi-variable Alexander polynomial so that we can apply the braid detection property of link Floer homology by Martin \cite[Corollary 2]{Martin:T26}. The link Floer homology and the multi-variable Alexander polynomial are related by \eqref{eq_HFL_Alex}. 

Suppose $L$ is a 2-component link, let $\Delta_L(x,y)\in \bZ[x,y,x^{-1},y^{-1}]$ be the multi-variable Alexander polynomial of $L$. 
Then $\Delta_L(x,y)$ is \emph{a priori} only well-defined up to a multiplication by $\pm x^ay^b$. It is possible to normalize the Alexander polynomial, for example, using Equation \eqref{eq_HFL_Alex}. However, the Alexander polynomial normalized by \eqref{eq_HFL_Alex} can be a Laurent polynomial with half-integer exponents. For our purpose, it is more convenient to take $\Delta_L(x,y)$ as Laurent polynomial with integer exponents, and therefore we will not normalize $\Delta_L(x,y)$.

For $f_1,f_2\in \bZ[x_1,x_1^{-1},\cdots,x_n,x_n^{-1}]$, we write $f_1\doteq f_2$ if and only if there exists a multiplicative unit $\epsilon$ such that $f_1=\epsilon\, f_2$. 

For $f\in \bZ[x_1,x_1^{-1},\cdots,x_n,x_n^{-1}]$, we use $\|f\|$ to denote the sum of the absolute values of the coefficients of $f$. 
By \eqref{eq_HFL_Alex}, we have 
$$\rank_{\bQ}\HFK(L;\bQ)=\rank_{\bQ}\HFL(L;\bQ)\ge \|(1-x)(1-y)\Delta_L(x,y)\|.$$

We need the following result.
\begin{Theorem}[{\cite{Torres}}]\label{Torres-multi-ALexander}
Suppose $L=K_1\cup K_2$ is a 2-component link with multi-variable Alexander polynomial $\Delta_L(x,y)$, where $x,y$ are the variables associated to $K_1,K_2$ respectively. 
Then
we have
$$
\Delta_L(x,1) \doteq \frac{1-x^l}{1-x} \Delta_{K_1}(x),
$$
where $\Delta_{K_1}(x)$ is the Alexander polynomial of $K_1$ and $l=\lk(K_1,K_2)$.
\end{Theorem}

From now on, let $L=K\cup U$ be a 2-component link such that
\begin{enumerate}
    \item  $U$ is an unknot,
    \item $K$ is either a trefoil or an unknot,
    \item the linking number $l=\lk(K,U)$ is positive.
\end{enumerate} 
Let $\Delta_L(x,y)$ be the multi-variable Alexander polynomial of $L$, where $x$, $y$ are the variables corresponding to $K$ and $U$ respectively. 
Define 
$$
F(x,y)=(1-x)(1-y)\Delta_L(x,y).
$$

By Theorem \ref{Torres-multi-ALexander}, we have
\begin{equation} 
\label{eqn_Delta(1,y)}
    \Delta_L(1,y)\doteq (1+y+\cdots+y^{l-1})\Delta_U(y)=1+y+\cdots+y^{l-1}.
\end{equation}
Write 
\begin{equation}
\label{eqn_defn_of_gm}
(1-y)\Delta_L(x,y)=:\sum_{m=-\infty}^{+\infty} g_m(x)y^m,
\end{equation}
then by definiton, 
\begin{equation}\label{eq_Fxy_gm}
F(x,y)= \sum_m (1-x)g_m(x)y^m.
\end{equation}
By \eqref{eqn_Delta(1,y)}, have
$$
(1-y)\Delta_L(1,y)\doteq (1-y) (1+y+\cdots+y^{l-1})=1-y^l.
$$
Therefore, after multiplying $\Delta_L(x,y)$ by $\pm y^a$, we may assume without loss of generality that
\begin{equation}
\label{eqn_value_of_g_at_1}
    g_0(1)=- g_l(1)=1,~g_m(1)=0~\text{for all}~m\neq 0, l.
\end{equation}

We establish the following two technical lemmas, which allow us to deduce topological properties of $L$ from the sequence of Laurent polynomials $\{g_m(x)\}_{m\in \bZ}$.

\begin{Lemma}\label{lem_supp_2_f-grading}
Let $L$, $\{g_m(x)\}_{m\in\bZ}$ be as above.
If $g_m(x)=0$ for all $m \neq 0, l$, then we have $l=1$.
\end{Lemma}
\begin{proof}
By the assumption and \eqref{eqn_defn_of_gm},  
$$
(1-y)\,\Delta_L(x,y)=g_0(x) +g_l(x)y^l.
$$
Plugging in $y=1$, we have $g_l(x)=-g_0(x)$, therefore
$$
\Delta_{L}(x,y)=\frac{(1-y^l)\,g_0(x)}{1-y}=(1+y+\cdots +y^{l-1})\,g_0(x),
$$
and hence
$$
\Delta_{L}(x,1)\doteq l\,g_0(x).
$$
On the other hand, by Theorem \ref{Torres-multi-ALexander}, 
$$
\Delta_{L}(x,1)\doteq (1+x+\cdots+x^{l-1})\Delta_{K}(x).
$$
Recall that $l$ is assumed to be positive. Comparing the two equations above, we have 
$$\frac{\Delta_K(x)}{l}\in\bZ[x,x^{-1}].$$
Since $\Delta_{K}(1)=\pm 1$, this implies $l=1$.
\end{proof}

Recall that for a Laurent polynomial $f$, we use $\|f\|$ to denote the sum of the absolute values of the coefficients of $f$. 
\begin{Lemma}\label{lem_Alex_4_degrees}
Let $L$, $\{g_m(x)\}_{m\in\bZ}$ be as above. Suppose the following two conditions hold: 
\begin{enumerate}
\item
There exists $k\in\bZ^+$, such that
$g_m(x)=0$ for all $m\neq 0,l,-k,l+k$,
\item 
$\|(1-x)g_0(x)\|=\|(1-x) g_l(x)\|=2$,
\end{enumerate}
then $l=1$ and $K$ is an unknot.
\end{Lemma}
\begin{proof}
By Condition (1),
\begin{equation}
\label{eqn_Delta(x,y)_expansion_4_terms_in_y}
    (1-y)\Delta_L(x,y) = g_{-k}(x)y^{-k}+g_0(x)+g_{l}(x)y^{l}+g_{l+k}(x)y^{l+k}.
\end{equation}

By Condition (2),
$\|(1-x)\,g_0(x)\|=2$.
Hence there exists an integer $s>0$ such that $(1-x)\,g_0(x)\doteq 1-x^s$, thus 
$$g_0(x) \doteq 1+\cdots +x^{s-1}.$$
By \eqref{eqn_value_of_g_at_1}, this implies $s=1$, therefore $g_0(x)\doteq 1$. Similarly $g_l(x)\doteq 1$. By  \eqref{eqn_value_of_g_at_1}, there exist integers $a,b$, such that
$$
g_0(x)=x^a, g_l(x)=-x^b.
$$
Plugging in $y=1$ to \eqref{eqn_Delta(x,y)_expansion_4_terms_in_y}, we obtain
\begin{equation}\label{eq_sum_g1234'}
g_{l+k}(x)=x^b-x^a-g_{-k}(x),
\end{equation}
and hence
\begin{align*}
(1-y)\Delta_{L}(x,y) &= g_{-k}(x) y^{-k} + x^a - x^b y^l + \big(x^b-x^a-g_{-k}(x)\big)y^{l+k}
\\
& = g_{-k}(x)(y^{-k}-y^{l+k}) + x^a(1-y^{l+k}) - x^b ( y^l-y^{l+k}).
\end{align*}
Therefore
\begin{align}
\Delta_{L}(x,1) & = \lim_{y\to 1}\Bigg(\frac{g_{-k}(x)(y^{-k}-y^{l+k}) + x^a(1-y^{l+k}) - x^b ( y^l-y^{l+k})}{1-y} \Bigg)
\nonumber
\\
& = (l+2k)\,g_{-k}(x)+(l+k)x^a-kx^b.
\label{eq_DeltaKU_x_1}
\end{align}
By \eqref{eq_DeltaKU_x_1}, if $\Delta_L(x,1)$ has more than two terms, then at least one of its coefficients is a multiple of $(l+2k)$.

On the other hand, by Theorem \ref{Torres-multi-ALexander},
$$
\Delta_{L}(x,1)\doteq (1+x+\cdots +x^{l-1}) \Delta_{K}(x),
$$
and hence (recall we have assumed that $K$ is either a trefoil or an unknot)
  \begin{equation}
  \label{eqn_compuation_of_Delta(x,1)}
    \Delta_{L}(x,1)\doteq
    \begin{cases}
      1+x+\cdots+x^{l-1} & \text{if $K$ is an unknot,} \\
      1-x+x^2 & \text{if $K$ is a trefoil and $l=1$,}   \\
       1+\displaystyle{\sum_{k=2}^{l-1} x^k}+x^{l+1} & \text{if $K$ is a trefoil and $l\ge 2$.}
    \end{cases}
  \end{equation}
In particular, all the coefficients of $\Delta_L(x,1)\in\bZ[x,x^{-1}]$ are $\pm 1$. Since $l+2k\ge 3$, the previous argument implies that $\Delta_L(x,1)$ has at most two terms, and hence there are three possibilities:
\begin{enumerate}
    \item $K$ is an unknot, $l=1$,
    \item $K$ is an unknot, $l=2$,
    \item $K$ is a trefoil, $l=2$.
\end{enumerate}
To eliminate the second and third possibilities, notice that in these cases,   \eqref{eq_DeltaKU_x_1} and \eqref{eqn_compuation_of_Delta(x,1)} yield 
$$
(2k+2) g_{-k}(x)+(k+2) x^a - k x^b\doteq x+1 \text{ or }x^3+1.
$$
By the assumptions, we have $k\in \bZ^+$, and hence
$$
 (k+2) x^a - k x^b \doteq x+1 \text{ or } x^3 + 1 \quad \text{mod } (2k+2) ,
$$
therefore we have
$$a\neq b,$$
and 
$$ k+2\equiv \pm 1, k \equiv \mp 1 \quad \text{mod } (2k+2),$$ 
which imply $k=1$, thus $g_{-k}(x) = -x^a$. This yields a contradiction to \eqref{eqn_value_of_g_at_1}.
\end{proof}

\begin{proof}[Proof of Proposition \ref{prop_rank_12_braid}]
Without loss of generality, we assume $l=\lk(K,U)>0$.
By \eqref{eqn_dim_assumption_HFK(L;Q)}, we have
\begin{equation}
\label{eqn_dim_assumption_HFL(L;Q)}
    \dim_\bQ \HFL(L;\bQ)= \dim_\bQ \HFK(L;\bQ)\le 12.
\end{equation}

For $a\in \frac12\bZ$, we use $\HFL(L,a;\bQ)$ to denote the component of $\HFL(L;\bQ)$ with degree $a$ on the Alexander grading associated to $U$. Let $\Delta_L(x,y)$, $F(x,y)$, $g_m(x)$ be as above, and we choose $\Delta_L(x,y)$ such that \eqref{eqn_value_of_g_at_1} holds.

Recall that by \eqref{eq_HFL_Alex}, the coefficients of $F(x,y)$ are the bi-graded Euler characteristics of $\HFL(L;\bQ)$.
Since $F(1,y)=0$, we have $\dim_\bQ\HFL(L,a;\bQ)$ is even for all $a$. By \eqref{eq_HFL_symmetry}, we have
\begin{equation}
\label{eqn_HFL(L,a;Q)_symmetric_at_zero}
    \dim_{\bQ}\HFL(L,a;\bQ)=\dim_{\bQ} \HFL(L,-a;\bQ).
\end{equation}

Since $0\neq \Delta_L(x,y)\doteq \Delta_L(x^{-1},y^{-1})$, there is a unique $(a,b)\in \frac12\bZ\times\frac12\bZ$, such that $\hat F(x,y):= x^ay^b F(x,y)$ satisfies $\hat F (x,y) = \pm \hat F(x^{-1},y^{-1})$. Write
$$
\hat F(x,y) = \sum_{m\in\frac12\bZ} \hat f_m(x) y^m,
$$
then by \eqref{eqn_value_of_g_at_1}, we have $\hat f_{l/2}(x)=\pm \hat f_{-l/2}(x^{-1})\neq 0$. Therefore by  \eqref{eq_HFL_Alex} and \eqref{eq_HFL_symmetry}, 
\begin{equation}
\label{eqn_non-vanishing_at_deg_l/2}
    \dim_{\bQ}\HFL(L,l/2;\bQ)= \dim_{\bQ}\HFL(L,-l/2;\bQ)\neq 0.
\end{equation}

Let $s\in \frac12\bZ$ be the maximum degree such that $\dim_{\bQ}\HFL(L,s;\bQ)\neq 0$, then $s\ge l/2>0$. Since $\dim_{\bQ}\HFL(L,s;\bQ)$ is even, by \eqref{eqn_dim_assumption_HFL(L;Q)} and \eqref{eqn_HFL(L,a;Q)_symmetric_at_zero}, we have 
    $$\dim_{\bQ}\HFL(L,s;\bQ)=2, 4,\text{ or }6.$$
    We discuss four cases.

{\bf Case 1.}  $\dim_{\bQ}\HFL(L,s;\bQ)=2$. 
By Proposition \ref{prop_braid_detection_HFL}, $K$ is a braid closure with axis $U$, therefore Case (1) of the proposition holds. 

{\bf Case 2.} $\dim_{\bQ}\HFL(L,s;\bQ) = 4$, and $s=\frac{l}{2}$.
By Proposition \ref{prop_HFL_Thurston_norm}, $U$ has a Seifert disk that intersects $K$ transversely at $l$ points, therefore this assumption contradicts Proposition \ref{prop_HFL_top_grading_parity}.

{\bf Case 3.} $\dim_{\bQ}\HFL(L,s;\bQ) =4 $, and $s>\frac{l}{2}$.
By \eqref{eqn_dim_assumption_HFL(L;Q)} and \eqref{eqn_non-vanishing_at_deg_l/2}, 
$$\dim_\bQ\HFL(L,\pm s;\bQ)=4,~ \dim_\bQ\HFL(L,\pm \frac{l}{2};\bQ)=2,$$
 and
$\HFL(L,a;\bQ)$ vanishes at all the other degrees. 
By \eqref{eq_HFL_Alex}, $\{g_m(x)\}_{m\in\bZ}$ satisfies the assumption of Lemma \ref{lem_Alex_4_degrees}, therefore $l=1$ and 
$K$ is an unknot, and hence Case (2) holds.

{\bf Case 4.}  $\dim_{\bQ}\HFL(L,s;\bQ)=6$. By
\eqref{eqn_dim_assumption_HFL(L;Q)} and \eqref{eqn_non-vanishing_at_deg_l/2}, we have 
$s=\frac{l}{2}$. By Proposition \ref{prop_HFL_Thurston_norm}, there is a Seifert disk of $U$ that intersects $K$ transversely at $l$ points.
By \eqref{eq_HFL_Alex}, $F(x,y)$ is supported at only two degrees in $y$, and hence $\{g_m(x)\}_{m\in\bZ}$ satisfies the assumption of Lemma \ref{lem_supp_2_f-grading}. By Lemma \ref{lem_supp_2_f-grading}, we have  $l=1$, therefore $U$ is a meridian of $K$. If $K$ is an unknot, then $L$ is the Hopf link, which satisfies Case (1). Otherwise, $K$ is a trefoil, and hence Case (3) holds.

\end{proof}

\section{Proof of the main theorem}
\label{sec_detection}

In this section, we prove that Khovanov homology detects the links $L_1$ and $L_2$ given by Figure \ref{fig_L1_and_L2}.

Recall that the internal grading of the Khovanov homology of a link $L$ is introduced in \cite[Section 2]{Kh-unlink} as $h-q$, where 
$h$ is the homological grading, and $q$ is the quantum grading. Following \cite{Kh-unlink}, we use $l$ to denote the internal grading.
 The next theorem is a special case of a more general result due to Batson and Seed.
\begin{Theorem}[{\cite[Corollary 4.4]{Kh-unlink}}]\label{Theorem_Kh_linking_number}
Suppose $L=K_1\cup K_2$ is a 2-component oriented link. Then we have
$$
\rank_{\mathbb{F}}^{l}\Kh(L;\mathbb{F})\ge \rank_{\mathbb{F}}^{l+2\lk(K_1,K_2)} (\Kh(K_1;\mathbb{F})\otimes \Kh(K_2;\mathbb{F}))
$$
for all $l\in\bZ$, where $\mathbb{F}$ is an arbitrary field and $\rank^k$ denotes the rank of the summand with internal grading $k$.
\end{Theorem}

Let $\bF=\mathbb{Z}/2$ from now on.
We have
\begin{equation}\label{eq_Kh_U2}
\Kh(U_2;\bF)=\bF_{(-2)}\oplus \bF_{(0)}^2\oplus \bF_{(2)},
\end{equation}
where $U_2$ is the 2-component unlink and the subscripts denote the internal gradings.
Let $T$ be the left-handed trefoil, and let $U$ be the unknot, we have
\begin{equation}\label{eq_Kh_TU}
\Kh(T;\bF)\otimes \Kh( U;\bF)=
\bF_{(0)}\oplus \bF_{(2)}^3 \oplus \bF_{(3)} \oplus\bF_{(4)}^3\oplus \bF_{(5)}^2 \oplus \bF_{(6)}\oplus \bF_{(7)}
\end{equation}
Let $\bar{T}$ be the right-handed trefoil, we have
\begin{equation}\label{eq_Kh_right_TU}
\Kh(\bar{T};\bF)\otimes \Kh( U;\bF)=
\bF_{(0)}\oplus \bF_{(-2)}^3 \oplus \bF_{(-3)} \oplus\bF_{(-4)}^3\oplus \bF_{(-5)}^2 \oplus \bF_{(-6)}\oplus \bF_{(-7)}
\end{equation}

Recall that the link $L_1=\text{L7n1}$ can be described as $\hat{\sigma}_1^3\cup U$ 
 where $\hat{\sigma}_1^3$ is the closure of the 2-braid $\sigma_1^3$ with axis unknot $U$, and we choose the orientation as given by Figure \ref{fig_L1_and_L2} and hence the linking number is $2$. The link $L_2$ is given by $T\cup U$ where $U$ is a meridian of $T$ and the orientation is chosen so 
that the linking number is $1$. 

We have
\begin{equation}\label{eq_Kh_L7n1}
\Kh(L_1;\bF) =
\bF_{(4)}\oplus \bF_{(6)}^3 \oplus \bF_{(7)} \oplus\bF_{(8)}^3\oplus \bF_{(9)}^2 \oplus \bF_{(10)}\oplus \bF_{(11)},
\end{equation} 
\begin{equation}\label{eq_Kh_T+m}
\Kh(L_2;\bF)= 
\bF_{(2)}\oplus \bF_{(4)}^3 \oplus \bF_{(5)} \oplus\bF_{(6)}^3\oplus \bF_{(7)}^2 \oplus \bF_{(8)}\oplus \bF_{(9)}.
\end{equation}

Besides the above links, we define the link $L_3=\widehat{\sigma_1\sigma_2}\cup U$,
which is the union of the closure of the 3-braid $\sigma_1\sigma_2$ and  
its axis unknot $U$. This is the torus link $T(2,6)$, which is denoted by 
L6a3 in the Thistlethwaite Link Table.
We pick the orientation properly so that the linking number is positive, then
\begin{equation}\label{eq_Kh_L6a3}
\Kh(L_3;\bF)=\bF_{(4)}\oplus \bF_{(6)}^2 \oplus \bF_{(7)} \oplus \bF_{(8)}^2 \oplus \bF_{(9)}^2 \oplus 
\bF_{(10)}^2 \oplus \bF_{(11)}\oplus \bF_{(12)}.
\end{equation}

We now prove Theorem \ref{thm_main_detection}.
\begin{reptheorem}{thm_main_detection}
Suppose $L=K_1\cup K_2$ is a 2-component oriented link and $i\in\{1,2\}$. If $\Kh(L;\bF)\cong \Kh(L_i;\bF)$ ($i=1,2$) as 
$l$-graded abelian groups, then  $L$ is isotopic to $L_i$ as oriented links.
\end{reptheorem}
\begin{proof}
Recall that $\bF=\bZ/2$. By the assumptions, we have 
\begin{equation}\label{eqn_assumption_Kh_rank_12}
    \rank_\bF\Kh(L;\bF)=12.
\end{equation}
By \cite[Corollary 1.7]{Dowlin}, we have
\begin{equation}\label{eq_HFK<Kh}
\rank_\bQ \HFK(L;\bQ)\le 2\rank_\bQ \Khr(L;\bQ)\le 2\rank_{\bZ/2} \Khr(L;{\mathbb{Z}/2})=12.
\end{equation}

Theorem \ref{Theorem_Kh_linking_number} yields
$$
\rank_\bF \Khr(K_i;\bF)=\frac12 \rank_\bF \Kh(K_i;\bF)\le \frac12 \frac{12}{2}=3.
$$
Therefore $K_i$ ($i=1,2$) is either the unknot or a trefoil according to \cite{KM:Kh-unknot,BS}.
By
Theorem \ref{Theorem_Kh_linking_number} again, we have at least one of $K_1$ and $K_2$ is the unknot. Without loss of generality we assume $K_2$ is an unknot, and we discuss two cases. 
\\

{\bf Case 1.} $K_1$ is also an unknot. We show that this is contradictory to the assumptions. In fact, 
Theorem \ref{Theorem_Kh_linking_number} and \eqref{eq_Kh_U2},
\eqref{eq_Kh_L7n1}, \eqref{eq_Kh_T+m}
imply that 
$l=\lk(K_1,K_2)$ 
is no less than $2$.  Therefore $K_1$ is the closure of an $l$-braid with axis $K_2$
by Proposition \ref{prop_rank_12_braid}. 
Switching the role of $K_1$ and $K_2$ we obtain that $K_2$ is the closure of an
$l$-braid with axis $K_1$. 
  By Proposition \ref{prop_linking_number_3}, we have $l\le 3$.
 If $l=2$, then $L=\hat{\sigma}_1^{\pm 1}\cup U$, which is the link L4a1 in the Thistlethwaite Link Table.
  This contradicts \eqref{eqn_assumption_Kh_rank_12} because 
 $\rank_\bF\Kh(\text{L4a1};\bF)=8$. If $l=3$, since the only 3-braid representations of the unknot 
 are given by $\sigma_1^{\pm 1}\sigma_2^{\pm 1}$ and $\sigma_1^{\pm 1}\sigma_2^{\mp 1}$, we further divide into two cases:

{\bf Case 1.1.}
 $L=\widehat{\sigma_1\sigma_2}\cup U=L_3$ or 
$L=\widehat{\sigma_1^{-1}\sigma_2^{-1}}\cup U=\bar{L}_3$. Recall that 
 $\Kh(L_3;\bF)$ 
 is given by \eqref{eq_Kh_L6a3}.
 Changing the orientation or taking the mirror image will shift the $l$-grading or change the sign of the $l$-grading, respectively. In any case, the $l$-graded Khovanov homology 
 of $L$
 cannot be isomorphic to $\Kh(L_i;\bF)$ ($i\in \{1,2\}$), contradicting the assumptions. 

{\bf Case 1.2.}
 $L=\widehat{\sigma_1\sigma_2^{-1}}\cup U$ or 
$L=\widehat{\sigma_1^{-1}\sigma_2}\cup U$. In this case $L$ is the link L6a2 (or its mirror image) 
in the Thistlethwaite Link Table. We have $\rank_\bF\Kh(L;\bF)=20$, which is not the same as $L_1$
and $L_2$. 

In conclusion, $K_1$ cannot be an unknot.
\\

{\bf Case 2.}
$K_1$ is a trefoil. There are two cases.

{\bf Case 2.1.}
 $K_1$ is the right-handed trefoil $\bar{T}$. Then Theorem \ref{Theorem_Kh_linking_number} and \eqref{eq_Kh_right_TU},
\eqref{eq_Kh_L7n1}, \eqref{eq_Kh_T+m} yield a contradiction. 

{\bf Case 2.2.}
$K_1$ is the left-handed trefoil $T$.


If $\Kh(L;\bF)\cong \Kh(L_1;\bF)$, then by Theorem \ref{Theorem_Kh_linking_number}, we have $\lk(K_1,K_2)=2$.
By 
Proposition \ref{prop_rank_12_braid}, the knot $K_1$ is the closure of a 2-braid in $S^3-N(K_2)$. 
A 2-braid representing the left-handed trefoil can only be $\sigma_1^3$.  Therefore
$L$ is isotopic $L_1$.

If $\Kh(L;\bF)\cong \Kh(L_2;\bF)$, then by Theorem \ref{Theorem_Kh_linking_number}, we have $\lk(K_1,K_2)=1$.
By Proposition \ref{prop_rank_12_braid}, the knot $K_2$ is a meridian of the left-handed trefoil 
$K_1$. Therefore $L$ is isotopic to $L_2$.
\end{proof}

\begin{remark}
The argument above gives an alternative proof of Martin's theorem that
Khovanov homology detects the torus link $T(2,6)$ \cite[Theorem 4]{Martin:T26}. In fact, the link $T(2,6)$  is detected by Case 1.1 in the argument above. 
\end{remark}

\bibliographystyle{amsalpha}
\bibliography{references}

\end{document}